\documentclass{amsart}
\usepackage{etex}

\usepackage{amssymb,amsmath,amsfonts,fancyhdr,amsthm,verbatim,stmaryrd}
\usepackage{graphicx,supertabular,paralist}
\usepackage[dvips,all]{xy}
\UseComputerModernTips
\usepackage{tikz}
\usepackage{pst-plot}
\usepackage{pst-node,pst-text,pst-3d,pstricks}
\usetikzlibrary{patterns}

\newtheorem{thm}{Theorem}[section]
\newtheorem{prop}[thm]{Proposition}
\newtheorem{lem}[thm]{Lemma}
\newtheorem{cor}[thm]{Corollary}

\theoremstyle{definition}

\newtheorem{eg}[thm]{Example}

\theoremstyle{remark}

\renewcommand{\emptyset}{\varnothing}

\renewcommand{\P}{\mathcal P}

\renewcommand{\phi}{\varphi}
\newcommand{\im}{\mathop{\mathrm{Im}}\nolimits}

\renewcommand{\hat}{\protect\widehat}

\renewcommand{\i}[1]{\mathfrak{#1}}

\newcommand{\m}{\i{m}}

\newcommand{\A}{\mathcal A}
\newcommand{\E}{\mathcal E}

\renewcommand{\L}{\mathcal L}

\renewcommand{\H}{\mathcal{H}}

\newcommand{\reg}{\mathop{\mathrm{reg}}\nolimits}

\newcommand{\pd}{\mathop{\mathrm{pd}}\nolimits}

\newcommand{\lcm}{\mathop{\mathrm{lcm}}\nolimits}

\renewcommand{\*}{\bullet}
\newcommand{\lto}{\mathop{\longrightarrow\,}\limits}
\newcommand{\xto}[1]{\overset{#1}{\lto}}

\title{Hypergraphs and Regularity of Square-free Monomial Ideals}
\author{Kuei-Nuan Lin}
\address{University of California, Riverside, Department of Mathematics, 900 University Ave. Riverside, CA, 92521}
\email{knlin@math.ucr.edu}
\author{Jason McCullough}
\address{Rider University, Department of Mathematics, 2083 Lawrenceville Road,
Lawrenceville, NJ 08648}
\email{jmccullough@rider.edu}


\begin{document}

\begin{abstract}We define a new combinatorial object, which we call a labeled hypergraph, uniquely associated to any square-free monomial ideal.  We prove several upper bounds on the regularity of a square-free monomial ideal in terms of simple combinatorial properties of its labeled hypergraph.  We also give specific formulas for the regularity of square-free monomial ideals with certain labeled hypergraphs.  Furthermore, we prove results in the case of one-dimensional labeled hypergraphs.
\end{abstract}

\maketitle

\section{Introduction}

Resolutions and invariants of square-free monomial ideals continue to be an active area of research.  It is well-known that given a homogeneous ideal, the regularity of its initial ideal with respect to any term order is an upper bound of the regularity of the ideal (cf. \cite[Theorem 22.9]{Peeva}). Moreover, the polarization of a monomial ideal has the same regularity as the original ideal and is square-free (cf. \cite[Theorem 21.10]{Peeva}). Therefore, understanding the regularity of square-free monomial ideals becomes extremely important.  Much of the existing literature on the subject involves associating to a square-free monomial ideal some combinatorial object.  The theory of simplicial complexes and Stanley-Reisner ideals is one example.  The theory of edge ideals of graphs is another.  Many authors have studied the regularity and Betti numbers of edge ideals of graphs, e.g.   \cite{DHS}, \cite{Katzman}, \cite{Kummini}, \cite{VanTuyl}, \cite{Woodroofe}.  Other authors have studied higher degree generalizations using hypergraphs and clutters \cite{Emtander}, \cite{HVT}, \cite{MKM} or simplicial complexes \cite{HVT2}, \cite{Terai} \cite{Zheng}.  See \cite{MV} for a nice survey of results on edge ideals of graphs and clutters.

 In this paper we focus on the regularity of arbitrary square-free monomial ideals.   We define a new combinatorial object associated to any square-free monomial ideal, which we call a \textit{labeled hypergraph}.  See Section~\ref{Slh} for a precise definition.  The unlabeled version of this construction was first used by Kimura et. al. \cite{KTY1} to study the arithmetical rank of square-free monomial ideals.  (See also \cite{KTY2}.)  By studying simple combinatorial properties of the labeled hypergraph associated to any square-free monomial ideal, we obtain bounds on the regularity of certain square-free monomial ideals in terms of the number of vertices and edges (Theorem~\ref{Tiso}, Proposition~\ref{Pmatch}).   In Theorem~\ref{Pfill}, we give our most general upper bound that applies to any square-free monomial ideal.  

In general, the regularity of a square-free monomial ideal may depend on the characteristic of the base field $K$ (cf. \cite[Section 5.3]{BH}), even in the case of the edge ideal of a simple graph \cite{Katzman}.  However, we give classes of labeled hypergraphs $\H$ whose associated square-free monomial ideal has regularity determined by the combinatorial properties of $\H$ and hence are independent of the base field (Theorem~\ref{Tsimple}, Corollary~\ref{Tmatch}).  These results show that the upper bounds mentioned above are tight for large classes of ideals.  However, even when the bounds are not tight, these results provide a very quick method to compute a rough upper bound for any square-free monomial ideal.

We note that by studying labeled hypergraphs of square-free monomial ideals, we do not impose restrictions on the degrees of the monomial generators, as opposed to the study of edge ideals of graphs, which are limited to monomial ideals generated in degree two.  Even most results regarding regularity of monomial ideals associated to hypergraphs or clutters requires some assumptions on the degrees of the generators or some other uniformity property.  We compare our bounds to two new general results of Dao-Schweig \cite{DS} and Ha-Woodroofe \cite{HW}.  We show that our results are incomparable to both.  See Example~\ref{egcompare}.

The rest of the paper is structured as follows: in Section~\ref{Sb}, we fix notation and summarize the relevant background results and terminology.  In Section~\ref{Slh}, we define the notion of a labeled hypergraph along with several properties of labeled hypergraphs that will be relevant in subsequent results.  Section~\ref{Smr} contains general upper bounds and more specific formulas for regularity of square-free monomial ideals in terms of labeled hypergraph data.  Finally, we focus on the case of one-dimensional labeled hypergraphs in Section~\ref{S1d}.

\section{Background and Notation}\label{Sb}

For the remainder of the paper, let $\A$ be a finite set of symbols, which will double as variables.  Let $R = K[\A]$ be a polynomial ring over a field $K$.  We give $R$ a standard graded structure, where all variables have degree one.  We write $R_i$ for the $K$-vector space of homogenous degree $d$ forms in $R$ so that $R = \bigoplus_{i \ge 0} R_i$.  We set $\m = \bigoplus_{i \ge 1} R_i$ to be the unique graded maximal ideal.  We use the notation $R(-d)$ to denote a rank-one free module with generator in degree $d$ so that $R(-d)_i = R_{i-d}$.  

Let $M$ be a finitely generated graded $R$-module.  We can compute the minimal graded free resolution of $M$: \[ 0 \to \bigoplus_{j} R(-j)^{\beta_{pj}(M)} \rightarrow \cdots \rightarrow \bigoplus_{j} R(-j)^{\beta_{2j}(M)} \rightarrow \bigoplus_{j} R(-j)^{\beta_{1j}(M)} \rightarrow \bigoplus_{j} R(-j)^{\beta_{0j}(M)},\]
  The minimal graded free resolution of $M$ is unique up to isomorphism.  Hence, the numbers $\beta_{ij}(M)$, called the \textit{graded Betti numbers} of $M$, are invariants of $M$.  Two coarser invariants measuring the complexity of this resolution are the \textit{projective dimension} of $M$, denoted $\pd(M)$, and the \textit{regularity} of $M$, denoted $\reg(M)$.  These can be defined as
\[\pd(M) = \max\{i\,:\,\beta_{ij}(M) \neq 0 \text{ for some } j\} \text{ and } \reg(M) = \max\{j - i\,:\,\beta_{ij}(M) \neq 0 \text{ for some } i\}.\]

We define labeled hypergraphs in Section~\ref{Slh} and use combinatorial properties to bound the regularity of square-free monomial ideals without assumptions on the the ideal itself.  Most of our proofs are inductive and we use the following standard facts about regularity often.  

\begin{lem} \label{Lreg} Let $I \subset R$ be an ideal and let $z \in R_d$.  Then
\begin{enumerate}
\item $\reg\left(\frac{R}{I}\right) \le \max\left\{\reg\left(\frac{R}{I:z}\right) + d, \reg\left(\frac{R}{(I,z)}\right)\right\}$
\item $\reg\left(\frac{R}{I:z}\right) \le \max\left\{\reg\left(\frac{R}{I}\right) - d, \reg\left(\frac{R}{(I,z)}\right) - d + 1\right\}$
\item $\reg\left(\frac{R}{(I,z)}\right) \le \max\left\{\reg\left(\frac{R}{I}\right), \reg\left(\frac{R}{I:z}\right) + d - 1\right\}$
\end{enumerate}
\end{lem}

\begin{proof} Apply Corollary 20.19 in \cite{Eisenbud} to the short exact sequence
\[0 \to \frac{R}{I:z}\left(- d\right) \xto{z} \frac{R}{I} \to \frac{R}{(I,z)} \to 0.\]
\end{proof}

The following is an immediate consequence.

\begin{cor} \label{Creg} Using the notation in Lemma~\ref{Lreg}, we have
\begin{enumerate}
\item If $\reg\left(\frac{R}{I:z}\right) + d \ge \reg\left(\frac{R}{(I,z)}\right) + 2$, then $\reg\left(\frac{R}{I}\right) = \reg\left(\frac{R}{I:z}\right) + d$.
\item If $\reg\left(\frac{R}{I:z}\right) + d \le \reg\left(\frac{R}{(I,z)}\right)$, then $\reg\left(\frac{R}{I}\right) = \reg\left(\frac{R}{(I,z)}\right)$.
\end{enumerate}
\end{cor}

We recall the definition of the Taylor Resolution \cite{Taylor}.  Let $I$ be a  monomial ideal with minimal monomial generators $f_1,\ldots,f_\mu$.  We define a complex $T_\*$ as follows: Set $T_1$ to be a free $R$-module with basis $e_1,\ldots,e_\mu$.  Let $T_i = \bigwedge^i T_1$ so that for all $F = \{j_1 < j_2 < \cdots < j_i\} \subset [\mu] := \{1,2,\ldots,\mu\}$, the elements $e_F = e_{j_1} \wedge \cdots \wedge e_{j_i}$ form a basis of $T_i$.  Then the differential $\partial_i:T_i \to T_{i-1}$ is determined by setting
\[
\partial_i(e_F) = \sum_{k = 1}^i (-1)^k \frac{\lcm(f_F)}{\lcm(f_{F\smallsetminus\{j_k\}})} e_{F\smallsetminus\{j_k\}},
\]
where $F = \{j_1 < j_2 < \cdots < j_i\} \subset [\mu]$ and $\lcm(f_F) = \lcm(f_{j_1},\ldots,f_{j_i})$.
For any monomial ideal $I$, the Taylor resolution is a graded free resolution of $R/I$ that is usually not minimal.  We begin our study of regularity in Section~\ref{Smr} by considering square-free monomial ideals whose Taylor resolutions are minimal.  In general, one obtains from the Taylor resolution of $I = (f_1\ldots,f_\mu)$ that 
\[\reg(R/I) \le \max\{\deg(\lcm(f_{j_1},\ldots,f_{j_i})) - i:1 \le i \le \mu \text{ and } 1 \le j_1 < j_2 < \cdots < j_i \le \mu\}.\]

\section{Labeled Hypergraphs}\label{Slh}

In this section we fix the notation and relevant definitions for labeled hypergraphs and their connection to square-free monomial ideals.  Contrary to the usual construction of a hypergraph associated to a square-free monomial ideal, where generators of the ideal correspond to edges, the vertices of a labeled hypergraph in our construction correspond to the generators and the edges correspond to the variables via divisibility.  

Let $V$  be a finite set.  Use $\P(V)$ to denote the power set of $V$.  A \textit{labeled hypergraph} (or just \textit{hypergraph}) $\H$ on $V$ with alphabet $\A$ is a tuple $(V, X, E,  \E)$, where $E: \A \to \P(V)$ is a function, $X = \{a \in \A:E(a) \neq \emptyset\}$ and $\E = \im(E)$.  We call the elements of $\E$ \textit{edges} (or \textit{hyperedges}) and often write $E_a = E(a)$.  For $F \in \E$, we call elements $a \in \A$ with $E_a = F$ the \textit{labels} of $F$.  Note that once $V$ and $E$ are set, $X$ and $\E$ are completely determined; however, we include them for notational convenience.  The number $|X|$ counts the number of labels appearing in $\H$ while $|\E|$ counts the number of distinct edges.  In other words, $|X| = \sum_{F \in \E} n_F$, where $n_F = |\{a \in A:E_a = F\}|$.

Consider a labeled hypergraph $\H = (V, X, E,  \E)$.  If $F, G \in \E$, we say that $F$ is a \textit{subface} of $G$ if $F \subseteq G$.  We use $|F|$ to denote the number of vertices in $F$ and call this number the \textit{size} of $F$.  We define the \textit{dimension of} $\H$ to be $\dim(\H) = \max\{|F| - 1 : F \in \E\}$.  We say that a vertex $v \in V$ is \textit{closed} if $\{v\} \in \E$; otherwise, we say $v$ is \textit{open}.   For any $v \in V$, we define the \textit{neighbors} of $v$ in $\H$ to be the set
\[N_{\H}(v) = \{w \in V:w \neq v \text{ and }\exists F \in \E \text{ such that } v, w \in F\}.\]
We say that  $\H$ has \textit{isolated open vertices} if for every open vertex $v$ and every $w \in N_\H(v)$, $w$ is closed.  An edge $F \in \E$ of $\H$ is called $\textit{simple}$ if $|F| \ge 2$ and $F$ has no proper subedges besides $\emptyset$.  We say that $\H$ has \textit{isolated simple edges} if every open vertex is contained in exactly one simple edge.  

Now let $I \subset R = K[\A]$ be a square-free monomial ideal with minimal monomial generating set $\{f_1,\ldots,f_\mu\}$.  The \textit{labeled hypergraph of} $I$ is the labeled hypergraph $\H(I) = (V, X, E,  \E)$, where $V = [\mu]$ and $E:\A \to \P([\mu])$ is defined by $E_a  = \{j:a \text{ divides } f_j\}$.  As above, $X = \{a \in \A:E_a\neq \emptyset\}$ and $\E = \im(E)$.  Note that $|X|$ counts the number of variables appearing in the minimal generators of $I$ and also counts the number of edges of $\H(I)$, with multiplicity.

\begin{eg}\label{eg1} Let $R=k[a,b,\ldots,y,z]$ be a polynomial ring over a field $k$.
Let $I=(f_{1},f_{2},f_{3},f_{4})$ a square-free monomial ideal where
$f_{1}=efh$, $f_{2}=aefgij$, $f_{3}=bchij$, $f_{4}=dghij$. The labeled hypergraph $\H(I) = (V, X, E,  \E)$, shown below in Figure~\ref{F1}, is defined by setting $V=[4]$,
$\A = \{a,b,\ldots,y,z\}$, $X =\{a,b,c,d,e,f,g,h,i,j\}$, 
$\mathcal{E}=\{\{2\},\{3\},\{4\},\{1,2\},\{2,4\},\{1,3,4\},\{2,3,4\}\}$, and
$E_{a}=\{2\}, E_{b}=E_c=\{3\}$, $E_{d}=\{4\}$,  $E_{e}=E_{f}=\{1,2\}$,
$E_{g}=\{2,4\}$, $E_{h}=\{1,3,4\}$ and $E_{i}=E_{j}=\{2,3,4\}$.  We label the vertices $1, 2, 3, 4$ in Figure~\ref{F1} for clarity.  For the rest of this paper, we omit such vertex labels.
\end{eg}

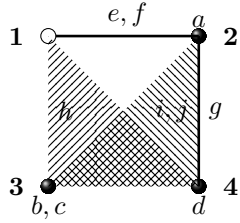
\begin{figure}[h]
\caption{$\H((efh,aefgij,bchij,dghij))$}\label{F1}
\begin{tikzpicture}
\usetikzlibrary{patterns}
\draw  [shape=circle] (-1,1) circle (.1) ;
\shade [shading=ball, ball color=black] (-1,-1) circle (.1) node [below] {$b,c$} ; 
\shade [shading=ball, ball color=black] (1,1) circle (.1) node [above ]  { $a$}  ;
\shade [shading=ball, ball color=black] (1,-1) circle (.1) node [below ] { $d$}  ;
\draw [line width=1pt  ] (-.9,1)--(1,1)  node [pos=.5, above ] { $e,f$ }; 
\draw [line width=1pt  ] (1,1)--(1,-1)  node [pos=.5, right ] { $g$ }  node [pos=.5, left ] { $i,j$ }; 
\path (-1,1)--(-2,1) node [pos=.2,left]{$\mathbf{1}$};
\path (-1,-1)--(-2,-1) node [pos=.2,left]{$\mathbf{3}$};
\path (1,-1)--(2,-1) node [pos=.2,right]{$\mathbf{4}$};
\path (1,1)--(2,1) node [pos=.2,right]{$\mathbf{2}$};
\path  (-1,1)--(-1,-1)  node [pos=.5, right ]{$h$};
\path [pattern=north east lines]   (-1,1)--(-1,-1)--(1,-1)--cycle;
\path [pattern=north west lines]     (1,1)--(1,-1)--(-1,-1)--cycle;
\end{tikzpicture}
\end{figure}

In \cite{KTY1}, Kimura et. al. introduced the unlabeled version of a hypergraph associated to a square-free monomial ideal where the vertices corresponded to minimal generators and edges corresponded to the variables of the ambient polynomial ring.  We note that this object is not detailed enough to study the regularity of square-free monomial ideals, since ideals with different regularity can have the same unlabeled hypergraph.  For example, the ideals $I = (ac,bc)$ and $J = (acd,bcd)$ both have hypergraphs on vertex set $V = \{1,2\}$ with edge set $\E = \{\{1\},\{2\},\{1,2\}\}$.  However, $\reg(R/I) = 1$ and $\reg(R/J) = 2$.  Since we will only talk about labeled hypergraphs for the remainder of the paper, we typically drop the word``labeled'' for brevity.

Following Kimura et. al \cite{KTY1}, we say that a hypergraph $\H = (V, X, E,  \E)$ is \textit{separated} if for every pair of vertices $v, w \in V$, there exist edges $F, G \in \E$ such that $v \in F \smallsetminus G$ and $w \in G \smallsetminus F$.  We note that $\H(I)$ is always a separated hypergraph for any square-free monomial ideal $I$, for if $\H(I)$ is not separated, one checks that the generators corresponding to the vertices of $\H(I)$ are not minimal.

Let $\H = (V, X, E,  \E)$ be a separated labeled hypergraph.  We define the square-free monomial associated to $\H$, denoted $I_\H$, as
the ideal generated by $\left\{\prod_{a \in \L_v} a \right\}_{v \in V}$ in the polynomial ring $K[\A]$, where $\L_v = \{a \in \A:v \in E_a\}$.\\

It is clear that $I_{\H(I)} = I$ for any square-free monomial ideal, and for any separated hypergraph, $\H(I_\H) = \H$, up to a permutation of the vertices.  We summarize this statement in the following proposition.

\begin{prop} There is a one-to-one correspondence
\begin{eqnarray*}
\left\{\parbox{1.5in}{Square-free monomials ideals}\right\} &\longleftrightarrow& \left\{\parbox{1.5in}{separated labeled hypergraphs up to vertex permutation}\right\}\\
I &\mapsto& \H(I)\\
I_\H &\mapsfrom& \H
\end{eqnarray*}
\end{prop}

We say that a labeled hypergraph $\H = (V, X, E,  \E)$ is \textit{saturated} if for all $v \in V$, $\{v\} \in \E$.  Note that if $I$ is a square-free monomial ideal, $\H(I)$ is saturated if and only if every minimal generator contains at least one variable not dividing any other generator.  We show in the next section that square-free monomial ideals with saturated hypergraphs have a combinatorial formula for their regularity.

\begin{eg}\label{eg2} Let $R = K[a,b,\ldots,y,z]$ and $I = (efhk,aefgij,bchij,dghij)$.  Then $\H(I)$ is saturated, as seen in Figure~\ref{F2}.  In contrast, the ideal from Example~\ref{eg1} is not saturated since $\{1\}$ is not an edge of the hypergraph.
\end{eg}

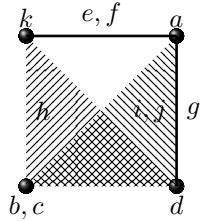
\begin{figure}[h]
\caption{$\H((efhk,aefgij,bchij,dghij))$}\label{F2}
\begin{tikzpicture}
\usetikzlibrary{patterns}
\shade [shading=ball, ball color=black] (-1,1) circle (.1)  node [above ]  { $k$};
\shade [shading=ball, ball color=black] (-1,-1) circle (.1) node [below] {$b,c$}; ;
\shade [shading=ball, ball color=black] (1,1) circle (.1) node [above ]  { $a$};
\shade [shading=ball, ball color=black] (1,-1) circle (.1) node [below ] { $d$};
\draw [line width=1pt  ] (-1,1)--(1,1)  node [pos=.5, above ] { $e,f$ }; 
\draw [line width=1pt  ] (1,1)--(1,-1)  node [pos=.5, right ] { $g$ }  node [pos=.5, left ] { $i,j$ }; 
\path  (-1,1)--(-1,-1)  node [pos=.5, right ]{$h$};
\path [pattern=north east lines]   (-1,1)--(-1,-1)--(1,-1)--cycle;
\path [pattern=north west lines]  (1,1)--(1,-1)--(-1,-1)--cycle;
\end{tikzpicture}
\end{figure}

\section{Main Results}\label{Smr}

Unless otherwise stated, we set $R = K[\A]$.  Let $I \subset R$ be a square-free monomial ideal with minimal monomial generators $f_1,\ldots,f_\mu$.  If $I$ is a complete intersection, then the Koszul complex $K_\*(f_1,\ldots,f_\mu)$ forms a minimal free resolution of $R/I$.  It is easy to see that $\reg(R/I) = \left(\sum_{i = 1}^\mu \deg(f_i)\right) - \mu = n - \mu$, where $n$ is the number of variables appearing among $f_1,\ldots,f_\mu$.  We first
show that ideals with saturated hypergraphs also satisfy this latter formula.

\begin{prop} \label{Psat} Let $I = (f_1,\ldots,f_\mu) \subset R$ be a square-free monomial ideal and let  $\H = \H(I) = (V, X, E,  \E)$.  Then the following are equivalent:
\begin{enumerate}
\item $\H$ is saturated.
\item The Taylor resolution of $R/I$ is minimal.
\end{enumerate}
In this case, we have
\[\reg(R/I) = |X| - |V| \text{ and } \pd(R/I) = |V|.\]
\end{prop} 

\begin{proof} 
Observe that $\lcm(f_1,\ldots,f_\mu) \neq \lcm(f_1,\ldots,\hat{f_i},\ldots,f_\mu)$ for all $1 \le i \le \mu$, where $\hat{f_i}$ denotes that the $i$th element is removed if and only if the Taylor resolution is minimal. (See e.g. Proposition 1 in \cite{Froberg}.) Since $I$ is square-free, it follows that every minimal generator $f_i$ has a variable which does not divide any of the other generators.  Hence $\H(I)$ is saturated.  The converse is also clear.

Since the Taylor resolution has length $\mu$, we have $\pd(R/I) = \mu = |V|$.  Keeping track of the internal degrees of the free modules in the Taylor resolution, we see that $T_\mu = R(-d)$, where $d$ is the number of distinct variables appearing in $I$.  Thus $\reg(R/I) = d - \mu = |X| - |V|$.
\end{proof}

\begin{eg} Consider the ideal $I = (efhk,cefgij,abhij,dghij)$ from Example~\ref{eg2}, and it has a saturated hypergraph $\H = \H(I) = (V,X,E,\E)$. By Proposition~\ref{Psat}, we get that $\reg(R/I) = |X| - |V| = 11 - 4 = 7$.
\end{eg}

Note that as a corollary, we obtain that the regularity of a square-free monomial ideals with a saturated hypergraph does not depend on the characteristic of the base field.  We wish to investigate how far the regularity of a square-free monomial ideal strays from this formula when its hypergraph is not saturated.  Observe that the regularity may be larger, if the projective dimension is smaller than $|V|$, or it may be much smaller, when the generators share several variables in common.  We first show that the above formula is an upper bound for square-free monomial ideals whose hypergraphs have isolated open vertices.

\begin{lem}\label{Liso} Let $I  \subset R = K[\A]$ be a square-free monomial ideal with minimal monomial generating set $\{f_1,\ldots,f_\mu\}$, and set $\H = (V,X,E,\E)$.  Suppose $\H$ has exactly $n$ isolated open vertices.  Further suppose $1, \mu \in V$ are such that $1$ is an isolated open vertex, $\mu$ is a closed vertex, and $\mu \in N_{\H}(1)$.  Write $J = (f_1, \ldots, f_{\mu - 1})$ and set $z = f_\mu$.  Finally set $\H' = \H(J:z) = (X',V',E',\E)$ and $\H'' = \H(J) = (V'',X'',E'',\E)$.
Then
\begin{enumerate}
\item $\H'$ and $\H''$ have at most $n$ isolated open vertices.
\item Either $1 \in V'$ is closed (in which case $\H'$ has fewer than $n$ isolated open vertices) or $|N_{\H'}(1)| \le |N_{\H}(1)| - 1$.
\item Either $1 \in V''$ is closed (in which case $\H''$ has fewer than $n$ isolated open vertices) or $|N_{\H''}(1)| \le |N_{\H}(1)| - 1$.
\end{enumerate}
Moreover, in all cases we have
\[|X''| \le |X| - 1,\quad |V''| = |V| - 1 \text{ and } |X'| \le |X| - \deg(z) - |V| + |V'| + 1\]
\end{lem}

Before we proceed with the proof, it may help the reader to keep the following example in mind:  
Let $R=k[a,b,\ldots,y,z]$ and let $I=(f_{1},\ldots,f_{5})$, where $f_{1}=ef$, $f_{2}=fg$, $f_{3}=bce$, $f_{4}=ab$,
$f_{5}=ade$.  Using the notation above, we set $\mu = 5$, $J=(f_{1},\ldots,f_{4})$, and $J:f_{5}=J:ade=(f,b)$.  The hypergraphs $\H = \H(I), \H' = \H(J:z)$ and $\H'' = \H(J)$ are pictured in Figure~\ref{Fdemo}.
\begin{figure}[h]
\caption{}\label{Fdemo}
\begin{tikzpicture}
\draw  [shape=circle] (0,0) circle (.1) ;
\draw  [shape=circle] (2,0) circle (.1) node [above] {$\mathbf{1}$};
\draw  [shape=circle] (-2,0) circle (.1) node [above] {$\mathbf{1}$};
\draw  [shape=circle] (-6,0) circle (.1) node [above] {$\mathbf{1}$};
\shade [shading=ball, ball color=black] (1,-1) circle (.1) node [below] {$c$}; 
\shade [shading=ball, ball color=black] (1,1) circle (.1) node [right] {$d$};
\path (1,1)--(1,2) node [pos=.2,above]{$\mathbf{5}$};
\shade [shading=ball, ball color=black] (3,0) circle (.1) node [below] {$g$}; 
\draw [line width=1pt  ] (1,1)--(0.07,0.07)  node [pos=.5, above ] { $a$}; 
\draw [line width=1pt  ] (1,-1)--(0.07,-.07)  node [pos=.5, left ] { $b$}; 
\draw [line width=1pt  ] (2.1,0)--(3,0)  node [pos=.5, above ] { $f$}; 
\path  (1.2,1)--(1.2,-1)  node [pos=.5, right ]{$e$};
\path [pattern=north west lines]  (1,1)--(1,-1)--(2,0)--cycle;
\draw  [shape=circle] (-2,0) circle (.1);
\shade [shading=ball, ball color=black] (-3,-1) circle (.1) node [below] {$c$}; 
\shade [shading=ball, ball color=black] (-4,0) circle (.1) node [above] {$a$}; 
\shade [shading=ball, ball color=black] (-1,0) circle (.1) node [below] {$g$}; 
\draw [line width=1pt  ] (-3,-1)--(-4,0)  node [pos=.5, above ] { $b$}; 
\draw [line width=1pt  ] (-3,-1)--(-2.07,-0.07)  node [pos=.5, left ] { $e$}; 
\draw [line width=1pt  ] (-1.9,0)--(-1,0)  node [pos=.5, above ] { $f$}; 
\shade [shading=ball, ball color=black] (-8,0) circle (.1) node [below] {$b$}; 
\shade [shading=ball, ball color=black] (-6,0) circle (.1) node [below] {$f$}; 
\path (1,-1)--(1,-2) node [pos=.7,below]{$\H(I)$};
\path (-3,-1)--(-3,-2) node [pos=.7,below]{$\H(J)$};
\path (-7,-1)--(-7,-2) node [pos=.7,below]{$\H(J:z)$};
\end{tikzpicture}
\end{figure}
The number of neighbors of vertex $1$ have been reduced, as in $\H(J)$, or vertex $1$ is now closed, as in $\H(J:z)$.  

\begin{proof}[Proof of Lemma~\ref{Liso}]
(1) For $i = 1,2,\ldots,\mu-1$, set $f'_i = f_i/\gcd(f_i,z)$.  Then $J:z = (f_1',\ldots,f_{\mu - 1}')$, although these need not all be minimal generators of $J:z$.  However, $f_1,\ldots,f_{\mu-1}$ is a minimal generating set of $J$.  Exactly $n$ of the vertices of $\H$ were open.  Hence at most $n$ vertices of $\H'$ and $\H''$ are open.  If either $\H'$ or $\H''$ have adjacent open vertices, then there would be some edge containing both.  Hence there would be some variable dividing both corresponding generators, which would still be true in $I$, contradicting that $I$ had isolated open vertices.

For (2) and (3), note that $\mu$ is no longer a neighbor of $1$ in $\H'$ or $\H''$ and no new neighbors of $1$ are created.  So $|N_{\H'}(1)| < |N_{\H}(1)|$ and $|N_{\H''}(1)| < |N_{\H}(1)|$.  Moreover, note that if $1$ is closed in either $\H'$ or $\H''$, then we have reduced the total number of open vertices by at least one.

Finally note that, since $\mu$ is closed in $\H$, there is at least one variable that divides only $f_\mu$.  Hence this variable does not appear in $J$ or $J:z$.  Further, $J$ has exactly one fewer generator than $I$; otherwise, one of $f_1,\ldots,f_{\mu-1}$ would be divisible by another, contradicting their minimality in $I$.  This gives us $|X''| \le |X| - 1$ and $|V''| = |V| - 1$.  

In $J:z$, note that none of the $\deg(z)$ variables in $z$ appear among the generators of $J:z$.  If $|V'| = |V|-1$, in which case $f_1',\ldots,f'_{\mu-1}$ is a minimal generating set of $J:z$, then we are done.  Otherwise, at least one of these generators is not minimal.  We claim that all such nonminimal generators $f_i'$ correspond to closed vertices $i$ of $\H$.  Suppose $f_i'$ is not a minimal generator of $J:z$ and suppose that $i$ is open in $\H$.  Then there is a vertex $j$ with $f_j'$ dividing $f_i'$.  It follows that $i$ and $j$ are neighbors in $\H$.  Since $i$ is open in $\H$, $j$ must be closed in $\H$, and hence in $\H'$ as well.  But this contradicts that $f_j'$ divides $f_i'$.    Therefore, any nonminimal generators lost in $J:z$ corresponded to closed vertices in $\H$.  For each such closed vertex, there is at least one variable dividing the corresponding generator in $I$ that does not appear in $J:z$.  It follows that $|X'| \le |X| - \deg(z) - |V| + |V'| + 1$.
\end{proof}

\begin{thm} \label{Tiso}Let $I \subset R$ be a square-free monomial ideal and suppose that $\H = \H(I) = (V, X, E,  \E)$ has only isolated open vertices.  Then
\[\reg(R/I) \le |X| - |V|.\]
\end{thm}

\begin{proof} We induct first on the number of isolated open vertices and then on the number
\[\min\{|N_{\H}(v)|:v \in V \text{ and $v$ is open in $\H$. }\}.\]
If $\H$ has no isolated open vertices, then $\H$ is saturated and the result follows from Proposition~\ref{Psat}.  Now assume $\H$ has exactly $n$ isolated open vertices and let $v \in V$ be such that $v$ is open in $\H$ and $|N_{\H}(v)|$ is minimal.   Without loss, we may assume $v = 1$.  Note that $|N_{\H}(1)| \ge 1$ since $v$ is open.   Pick a vertex $\mu \in N_{\H}(1)$,  which is necessarily closed.  Let $z = f_\mu$.  Let $J = (f_1,f_2,\ldots,f_{\mu - 1})$ and let $\H' = \H(J:z) = (X', V', E', \A', \E')$ and $\H'' = \H(J) = (X'',V'',E'',\A'',\E'')$.  By Lemma~\ref{Liso}, $\H'$ and $\H''$ either have fewer isolated open vertices or we have reduced the minimal number of adjacent vertices to an isolated open vertex.  So we may assume by induction that
\[\reg(R/J) \le |X''| - |V''| \text{ and } \reg(R/(J:z)) \le |X'| - |V'|.\]
Again by Lemma~\ref{Liso}, we have that
\[|X''| \le |X| - 1,\quad |V''| = |V| - 1 \text{ and } |X'| \le |X| - \deg(z) - |V| + |V'| + 1\]
Therefore,
\[\reg(R/J) \le |X''| - |V''| \le \left(|X| - 1\right) - \left(|V| - 1\right) = |X| - |V|,\]
and
\[\reg(R/(J:z)) \le |X'| - |V'| \le \left(|X| - \deg(z) - |V| + |V'| + 1\right) - |V'| = |X| - |V| - \deg(z) + 1.\]
Finally, it follows from Lemma~\ref{Lreg} that $\reg(R/I) \le |X| - |V|$.
\end{proof}

\begin{eg} Let $I = (efh,cefgij,abhij,dghij)$, as in Example~\ref{eg1}.  Note that $\H(I) = (V,X,E,\E)$ has a single isolated open vertex.  Therefore, by Theorem~\ref{Tiso}, we have
\[\reg(R/I) \le |X| - |V| = 10 - 4 = 6.\]
In this case, one computes that $\reg(R/I) = 6$.  
     So the bound in Theorem~\ref{Tiso} is tight in some cases.
\end{eg}

\begin{eg} The bound in Theorem~\ref{Tiso} is not always tight when $I$ has edges with multiple labels.  Let $R=k[a,b,\ldots,x,y,z]$ a polynomial
ring over a field $k$. Let $I=(f_{1},f_{2},f_{3})$ a square-free
monomial ideal, where $f_{1}=ab$, $f_{2}=acd$, $f_{3}=bef$.  $\H = \H(I) = (V,X,E,\E)$ has exactly one isolated open vertex at $2 \in V$.  See Figure~\ref{F4}.  Theorem~\ref{Tiso} implies that $\reg(R/I) \le |X| - |V| = 6 - 3 = 3$; however, $\reg(R/I) = 2$.  
\end{eg}

\begin{figure}[h]
\caption{$\H((ab,acd,bef))$}\label{F4}
\begin{tikzpicture}
\usetikzlibrary{patterns}
\draw  [shape=circle] (0,1) circle (.1) ;
\shade [shading=ball, ball color=black] (-1,-1) circle (.1) node [below] {$c,d$}; 
\shade [shading=ball, ball color=black] (1,-1) circle (.1) node [below ]  { $e,f$};
\draw [line width=1pt  ] (.07,.93)--(1,-1)  node [pos=.5, right ] { $b$ }; 
\draw [line width=1pt  ] (-.07,.93)--(-1,-1)  node [pos=.5, left] { $a$ }; 
\end{tikzpicture}
\end{figure}

\begin{eg} In general, we need the hypothesis about isolated open vertices.  Let $I = (ab, bc, ac)$.  Then $\H = \H(I) = (V,X,E,\E)$, pictured in Figure~\ref{F5}, does not have isolated open vertices.  One quickly computes that $|X| - |V| = 3 - 3 = 0$ and $\reg(R/I) = 1$.  
\end{eg}

\begin{figure}[h]
\caption{$\H((ab,ac,bc))$}\label{F5}
\begin{tikzpicture}
\usetikzlibrary{patterns}
\draw  [shape=circle] (0,0) circle (.1) ;
\draw  [shape=circle] (-2,-2) circle (.1) ;
\draw  [shape=circle] (2,-2) circle (.1) ;
\draw [line width=1pt  ] (-.07,-.07)--(-1.93,-1.93)  node [pos=.5, above ] { $a$ }; 
\draw [line width=1pt  ] (-1.9,-2)--(1.9,-2)  node [pos=.5, below ] { $b$ } ;
\draw [line width=1pt  ] (1.93,-1.93)--(.07,-.07)  node [pos=.5, right ] { $c$ } ;
\end{tikzpicture}
\end{figure}

If we want to have a more general upper bound like that in the previous theorem, we can modify the hypergraph of $I$ until we are in the isolated open vertices situation and modify the upper bound accordingly.  We separate the inductive step of the proof in the following lemma.

\begin{lem}\label{Lfill} Let $I  \subset R$ be a square-free monomial ideal with minimal monomial generators $f_1,\ldots,f_\mu$, and let $\H = \H(J) = (V,X,E,\E)$.  Suppose $1 \in V$ is an open vertex.  Let $x \in \A - X$ and set $J = (f_1x,f_2,\ldots,f_\mu)$.  Let $\H' = \H(I) = (X',V',E',\E')$ and let $\H'' = \H((x,J)) = (V'',X'',E'',\E'')$.  Then  
\[|V| = |V'| = |V''| \text{ and  } |X'| = |X''| = |X| + 1.\]
Moreover, vertex $1$ is closed in both $\H'$ and $\H''$.
\end{lem}

\begin{proof} First note that $I = J:x$.  By the assumptions above, we obtain $\H'$ by filling in the open vertex $1$ in $\H$ and setting $E'_x = \{1\}$.  $\H''$ is obtained from $\H'$ by replacing any edge $F \in \E'$ with $F\smallsetminus\{1\}$.  The generator $x$ replaces $f_1x$ and all other generators remain minimal.  The conclusion then follows easily.
\end{proof}

We can now give a general upper bound on the regularity of all square-free monomial ideals.

\begin{thm}\label{Pfill} Let $I \subset R$ be a square-free monomial ideal.  Let $\H = \H(I) = (V, X, E,  \E)$.  Suppose that if we can add edges of size $1$, say $\{v_1\},\ldots,\{v_t\}$ to $\E$, where $v_1,\ldots,v_t \in V$, we obtain a new hypergraph with isolated open vertices.  Then
\[\reg(R/I) \le |X| - |V| + t.\]
\end{thm}

\begin{proof} We induct on $t$.  If $t =  0$, then $\H$ has isolated open vertices and we rely on Theorem~\ref{Tiso}.  So assume $t > 0$ and the result holds for smaller cases.  Write $I = (f_1,\ldots,f_\mu)$.  Let $v = v_t$ and pick $x \in \A - X$.  Define $J = (f_1,\ldots,f_vx,\ldots,f_\mu)$ so that $J:x = I$.  Set $\H' = \H(J) = (X',V',E',\E)$ and $\H'' = \H(x,J) = (X'',V'',E'',\E'')$.  By Lemma~\ref{Lfill}, we have
\[|V| = |V'| = |V''| \text{ and } |X'| = |X''| = |X| + 1.\]
Moreover, we may isolate the open vertices of $\H'$ and $\H''$ by filling in at most $t - 1$ open vertices in each.  By induction, we have
\[\reg(R/(J)) \le |X'| - |V'| + t -1 = |X| - |V| + t,\]
and
\[\reg(R/(x,J)) \le |X''| - |V''| + t - 1 = |X| - |V| + t.\]
Finally, by Lemma~\ref{Lreg}, we get
\[\reg(R/I) = \reg(R/(J,x)) \le |X| - |V| + t.\]
\end{proof}

\begin{eg} Let $R = K[a,b,\ldots,y,z]$  Set 
\[I = (di, ade, bij, fgij, efg, jh, ch) \text{ and } J = (di, ade, bij, fgijk, efg, jh, ch).\]
The hypergraphs $\H(I) = (V,X,E,\E)$ and $\H(J) = (V',X',E',\E')$ are pictured in Figure~\ref{F6}.  Note that $\H(J)$ has isolated open vertices.  Hence $\reg(R/J) \le |X'| - |V'| = 11 - 7 = 4$ by Theorem~\ref{Tiso}.  We obtain $\H(J)$ from $\H(I)$ by filling in the vertex corresponding to '$k$' in $\H(J)$.  By Theorem~\ref{Pfill}, $\reg(R/I) \le |X| - |V| + 1 = 10 - 7 + 1 = 4$.  One can check that $\reg(R/I) = 4$ in this case.
\end{eg}

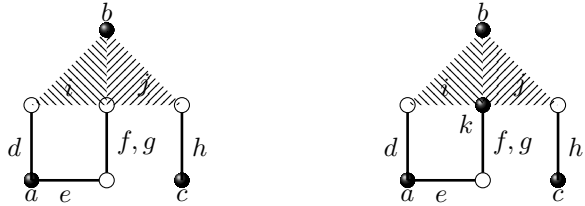
\begin{figure}[h]
\caption{$\H((di, ade, bij, fgij, efg, jh, ch))$ \text{ and } $\H((di, ade, bij, fgijk, efg, jh, ch))$.}\label{F6}
\begin{tikzpicture}

\draw  [shape=circle] (-1,0) circle (.1) ;
\draw  [shape=circle] (0,0) circle (.1) ;
\draw  [shape=circle] (1,0) circle (.1) ;
\draw  [shape=circle] (0,-1) circle (.1) ;
\shade [shading=ball, ball color=black] (0,1) circle (.1) node [above] {$b$} ; 
\shade [shading=ball, ball color=black] (-1,-1) circle (.1) node [below ]  { $a$}  ;
\shade [shading=ball, ball color=black] (1,-1) circle (.1) node [below ] { $c$}  ;
\draw [line width=1pt  ] (-1,-.1)--(-1,-1)  node [pos=.5, left ] { $d$ }; 
\draw [line width=1pt  ] (0,-.1)--(0,-.9)  node [pos=.5, right ] { $f,g$ };
\draw [line width=1pt  ] (-1,-1)--(-.1,-1)  node [pos=.5, below ] { $e$ };
\draw [line width=1pt  ] (1,-.1)--(1,-1)  node [pos=.5, right ] { $h$ };
\path  (-1,0)--(0,0)  node [pos=.5, above ]{$i$};
\path  (0,0)--(1,0)  node [pos=.5, above ]{$j$};
\path [pattern=north west lines]   (0,1)--(-1,0)--(0,0)--cycle;
\path [pattern=north east lines]     (0,1)--(0,0)--(1,0)--cycle;
\draw  [shape=circle] (4,0) circle (.1) ;
\draw  [shape=circle] (6,0) circle (.1) ;
\draw  [shape=circle] (5,-1) circle (.1) ;
\shade [shading=ball, ball color=black] (5,0) circle (.1) node [below left] {$k$} ; 
\shade [shading=ball, ball color=black] (5,1) circle (.1) node [above] {$b$} ; 
\shade [shading=ball, ball color=black] (4,-1) circle (.1) node [below ]  { $a$}  ;
\shade [shading=ball, ball color=black] (6,-1) circle (.1) node [below ] { $c$}  ;
\draw [line width=1pt  ] (4,-.1)--(4,-1)  node [pos=.5, left ] { $d$ }; 
\draw [line width=1pt  ] (5,-.1)--(5,-.9)  node [pos=.5, right ] { $f,g$ };
\draw [line width=1pt  ] (4,-1)--(4.9,-1)  node [pos=.5, below ] { $e$ };
\draw [line width=1pt  ] (6,-.1)--(6,-1)  node [pos=.5, right ] { $h$ };
\path  (4,0)--(5,0)  node [pos=.5, above ]{$i$};
\path  (5,0)--(6,0)  node [pos=.5, above ]{$j$};
\path [pattern=north west lines]   (5,1)--(4,0)--(5,0)--cycle;
\path [pattern=north east lines]     (5,1)--(5,0)--(6,0)--cycle;
\end{tikzpicture}
\end{figure}

We now show that the bound in Theorem~\ref{Pfill} is tight for a large class of hypergraphs.  We first need the following lemma.

\begin{lem}\label{Lsimple} Let $I \subset R$ be a square-free monomial ideal.  Suppose $\H(I) = (V,X,E,\E)$ has $n$ isolated simple edges, $F_1,F_2,F_3\ldots,F_n \in \E$.  Set $z = \prod_{\substack{a \in X\\E_a = F_1}} a$, $\H' = \H(I:z) = (X',V',E',\E)$ and $\H'' = \H((I,z)) = (X'',V'',E'',\E'')$.  Then
\begin{enumerate}
\item $\H'$ has at most $\sum_{i = 1}^n |F_n|$ total open vertices corresponding to the open vertices in the simple edges of $\H$, including at most $|F_1|$ isolated open vertices corresponding to the vertices of $F_1$.
\item $\H''$ has exactly $n-1$ isolated simple edges corresponding to $F_2,\ldots,F_n$ and no other open vertices.
\end{enumerate}
Moreover, we have
\[|X'| \le |X| - \deg(z) - |V| + |V'|, \quad |X''| = |X|  \text{ and } |V''| = |V| - |F_1| + 1.\]
\item 
\end{lem}

\begin{proof} First consider $\H'$.  All $\deg(z)$ variables that appeared in $z$ are missing in $I:z$.    If $\H$ with $F_1$ removed is separated, then the remaining generators of $I$ minimally generate $I:z$ and we are done.  If the resulting hypergraph is not separated, then some of the generators corresponding to neighbors of the open vertices of $F_1$ may not be minimal and can be discarded.  Note that for every closed vertex corresponding to a minimal generator of $I$ that we discard, we remove at least one variable.  It follows that $|X'| \le |X| - \deg(z) - |V| + |V'|$.

In $\H''$, we add one additional generator $z$.  Any of the generators corresponding to the vertices of $F_1$ are no longer minimal and deleted.  Since the vertices in $F_1$ were open and since $F_1$ was simple, $|X''| = |X|$.  All other generators remain minimal, and so $|V| = |V''| - |F_1| + 1$.
\end{proof}

\begin{thm}\label{Tsimple} Let $I \subset R$ be a square-free monomial ideal and suppose that $\H(I) = (V,X, E,  \E)$ has isolated simple edges.  Then
\[\reg(R/I) = |X| - |V| + \sum_{\substack{F \in \E\\F \text{ simple}}} (|F| - 1).\]
\end{thm}

\begin{proof} We induct on the number $n$ of isolated simple edges of $\H$.  If $n = 0$, then $\H$ is saturated and the result follows from Proposition~\ref{Psat}.  

Now assume $n > 0$, let $F_1,\ldots,F_n$ be those isolated simple edges of $\H$, and assume the result holds for hypergraphs with at most $n-1$ isolated simple edges.  Let
\[z = \prod_{\substack{a \in X\\E_a = F_1}} a.\]
Let $\H' = \H(I:z) = (V',X',E',\E')$ and $\H'' = \H((I,z)) = (V'',X'',E'',\E'')$.  By Lemma~\ref{Lsimple},
\[|X'| \le |X| - \deg(z) - |V| + |V'|.\]
Moreover, we need to fill in at most $\sum_{i = 2}^n (|F_i| - 1)$ open vertices in $\H'$ to create a graph with isolated open vertices.  By Theorem~\ref{Pfill},
\[\reg(R/(I:z)) \le |X'| - |V'| + \sum_{i = 2}^n (|F_i| - 1) \le |X| - |V| + \sum_{i = 2}^n (|F_i| - 1) - \deg(z).\]
Also by Lemma~\ref{Lsimple}, $\H''$ has exactly $n-1$ isolated simple edges of sizes $|F_2|, |F_3|,\ldots, |F_n|$ and no other open vertices.  We also have
\[|X''| = |X|  \text{ and } |V''| = |V| - |F_1| + 1.\]
By induction, we have
\[\reg(R/(I,z)) = |X''| - |V''| + \sum_{i = 2}^n \left(|F_i| - 1\right) = |X| - |V| + \sum_{i = 1}^n\left(|F_i| - 1\right).\]
Since $|F_1| \ge 2$, $\reg(R/(I:z)) + \deg(z) \le \reg(R/(I,z))$.  Thus by Corollary~\ref{Creg},
\[\reg(R/I) = \reg(R/(I,z)) =  |X| - |V| + \sum_{i = 1}^n\left(|F_i| - 1\right).\]
\end{proof}

\begin{eg}
Let $R=K[a,b,\ldots,y,z]$. Let $I=(f_{1},\ldots,f_{7})$, where $f_{1}=ab$, $f_{2}=bcdef$, $f_{3}=ac$, $f_{4}=eg$,
$f_{5}=fg$, $f_{6}=gh$, $f_{7}=hi$.  Set $\H = \H(I) = (V, X, E,  \E)$.  Then $\H$, pictured in Figure~\ref{F3}, has exactly 2 isolated simple edges: $E_a$ and $E_g$.  Therefore
\[\reg(R/I) = |X| - |V| + \sum_{\substack{F \in \E\\F \text{ simple}}} (|F| - 1) = 9 - 7 + (2 - 1) + (3 - 1) = 5.\]
\end{eg}

\begin{figure}[h]
\caption{$\H((ab,bcdef,ac,eg,fg,gh,hi))$}\label{F3}
\begin{tikzpicture}
\usetikzlibrary{patterns}
\draw  [shape=circle] (-1,1) circle (.1) ;
\draw  [shape=circle] (1,1) circle (.1) ;
\draw  [shape=circle] (1,-1) circle (.1) ;
\draw  [shape=circle] (2,0) circle (.1) ;
\draw  [shape=circle] (-1,-1) circle (.1) ;
\shade [shading=ball, ball color=black] (0,0) circle (.1) node [below] {$d$}; 
\shade [shading=ball, ball color=black] (3,0) circle (.1) node [below] {$i$}; 
\draw [line width=1pt  ] (-.93,.93)--(0,0)  node [pos=.5, above ] { $b$ }; 
\draw [line width=1pt  ] (-1,.9)--(-1,-.9)  node [pos=.5, left ] { $a$ }; 
\draw [line width=1pt  ] (-.93,-.93)--(0,0)  node [pos=.5, below ] { $c$ }; 
\draw [line width=1pt  ] (.93,.93)--(0,0)  node [pos=.5, above ] { $e$ }; 
\draw [line width=1pt  ] (2.1,0)--(3,0)  node [pos=.5, above ] { $h$ }; 
\draw [line width=1pt  ] (.93,-.93)--(0,0)  node [pos=.5, below ] { $f$ }; 
\path  (1.2,1)--(1.2,-1)  node [pos=.5, right ]{$g$};
\path [pattern=north west lines]  (1,1)--(1,-1)--(2,0)--cycle;
\end{tikzpicture}
\end{figure}
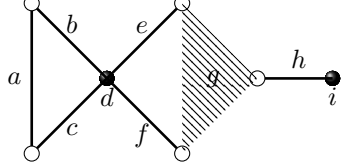

The results in this section seem to be incomparable to other known bounds on the regularity of square-free monomial ideals.  To our knowledge, the best combinatorial upper bounds on the regularity of a square-free monomial ideal with no other assumptions are the following results by Dao-Schweig and Ha-Woodroofe.

\begin{thm}[{\cite[Remark 6.4]{DS}}]\label{DS} If $\mathcal{C}$ is a clutter, write $\reg(\mathcal{C})$ to denote the (Castelnuovo-Mumford) regularity of $I(\mathcal{C})$.  Then
\[\reg(\mathcal{C}) \le |V(\mathcal{C})| - \epsilon(\mathcal{C}^\vee).\]
\end{thm}
Here $\mathcal{C}^\vee$ denotes the clutter of the Alexander dual of $I(\mathcal{C})$ and $\epsilon(-)$ denotes the edgewise domination parameter of a clutter.  (See \cite[Definition 3.1]{DS}.)

\begin{thm}[{\cite[Theorem 1.2]{HW}}]\label{HW}  Let $\mathcal{H}$ be a simple hypergraph with edge ideal $I \subseteq R$, and let $\{E_1,\ldots,E_c\}$ be a $2$-collage in $\H$.  Then
\[\reg(R/I) \le \sum_{i = 1}^c (|E_i| - 1).\]
\end{thm}
See \cite{HW} for the definition of a $2$-collage.  

It is worth noting that the notions of \textit{simple hypergraph} in \cite{HW} and of a \textit{clutter} in \cite{DS} coincide with each other.  When referring to this object below, we will use the term clutter so as to distinguish it from the labeled hypergraphs we have been studying.

\begin{eg}\label{egcompare} Here we give two examples that show that the bound in Theorem~\ref{Pfill} is stronger than the two above results for some ideals and weaker for others.  

Consider the ideal $I = (abc, def, adg, beg)$.  The labeled hypergraph $\H(I)$ is pictured in Figure~\ref{f8}.

\begin{figure}[h]
\caption{$\H((abc, def, adg, beg))$}\label{f8}
\begin{tikzpicture}
\usetikzlibrary{patterns}
\draw  [shape=circle] (-1,1) circle (.1) ;
\draw  [shape=circle] (-1,-1) circle (.1) ;
\shade [shading=ball, ball color=black] (0,0) circle (.1) node [right] {$f$}; 
\shade [shading=ball, ball color=black] (-2,0) circle (.1) node [left] {$c$}; 
\draw [line width=1pt  ] (-.93,.93)--(0,0)  node [pos=.5, above ] { $e$ }; 
\draw [line width=1pt  ] (-1.07,.93)--(-2,0)  node [pos=.5, above ] { $b$ }; 
\draw [line width=1pt  ] (-1,.9)--(-1,-.9)  node [pos=.5, left ] { $g$ }; 
\draw [line width=1pt  ] (-.93,-.93)--(0,0)  node [pos=.5, below ] { $d$ }; 
\draw [line width=1pt  ] (-1.07,-.93)--(-2,0)  node [pos=.5, below ] { $a$ }; 
\end{tikzpicture}
\end{figure}
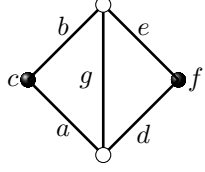

Since $\H(I)$ has only two open vertices, it follows from Theorem~\ref{Pfill} that 
\[\reg(R/I) \le |X| - |V| + 1 = 7 - 4 + 1 = 4.\]
(In fact, since this labeled hypergraph has one isolated simple edge of size 2, Theorem~\ref{Tsimple} shows that $\reg(R/I) = 4$.)
The clutter $\mathcal{C}$ associated to $I$ is pictured in Figure~\ref{f9}.

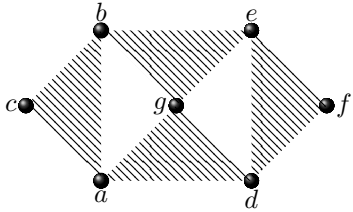
\begin{figure}[h]
\caption{$\mathcal{C}((abc, def, adg, beg))$}\label{f9}
\begin{tikzpicture}
\usetikzlibrary{patterns}
\draw  [shape=circle] (-1,1) circle (.1) ;
\draw  [shape=circle] (-1,-1) circle (.1) ;
\shade [shading=ball, ball color=black] (-1,1) circle (.1) node [above] {$b$}; 
\shade [shading=ball, ball color=black] (-1,-1) circle (.1) node [below] {$a$}; 
\shade [shading=ball, ball color=black] (1,-1) circle (.1) node [below] {$d$}; 
\shade [shading=ball, ball color=black] (1,1) circle (.1) node [above] {$e$}; 
\shade [shading=ball, ball color=black] (0,0) circle (.1) node [left] {$g$}; 
\shade [shading=ball, ball color=black] (-2,0) circle (.1) node [left] {$c$}; 
\shade [shading=ball, ball color=black] (2,0) circle (.1) node [right] {$f$}; 
\path [pattern=north west lines]  (1,1)--(1,-1)--(2,0)--cycle;
\path [pattern=north west lines]  (1,1)--(-1,1)--(0,0)--cycle;
\path [pattern=north west lines]  (1,-1)--(-1,-1)--(0,0)--cycle;
\path [pattern=north west lines]  (-1,1)--(-1,-1)--(-2,0)--cycle;
\end{tikzpicture}
\end{figure}

It is easy to check that the set of edges $\mathcal{E}(\mathcal{C})$ is itself a $2$-collage of $\mathcal{C}$.  So by Theorem~\ref{HW},
\[\reg(R/I) \le \sum_{E \in \mathcal{E}(\mathcal{C})} (|E| - 1) = 8.\]

The Alexander dual of $I$ is
\[I^\vee = (bd, ae, cde, abf, adg, cdg, beg, ceg, afg, bfg, cfg).\]
Since every variable appears in at least one generator with the variable $g$, every vertex in $\mathcal{C}^\vee$ is a neighbor of the vertex corresponding to $g$.  It follows that any one edge in $\mathcal{C}^\vee$ containing $g$ forms an edgewise dominant set and $\epsilon(\mathcal{C}^\vee) = 1$.  Hence, by Theorem~\ref{DS},
\[\reg(R/I) = \reg(I) - 1 = \reg(\mathcal{C}) - 1 \le |V(\mathcal{C}| - \epsilon(\mathcal{C}^\vee) - 1 = 7 - 1 - 1 = 5.\]
Hence the regularity bound from Theorem~\ref{Pfill} is sharper than both of these results in this instance.  However, one checks that if $I = (abc, abd, acd, bcd)$, then Theorem~\ref{Pfill} yields $\reg(R/I) \le 5$ while both Theorem~\ref{DS} and Theorem~\ref{HW} yield $\reg(R/I) \le 2$.  
\end{eg}

\section{One-Dimensional Hypergraphs}\label{S1d}

The following observation was made by Dao et.~al. in the proof of Lemma 2.10 in \cite{DHS}.  The argument is based on Lemma 2.2 in \cite{Kummini}.

\begin{lem}\label{Ldhs} Let $I \subset R$ be a square-free monomial ideal and let $x$ be a variable appearing in the minimal generators of $I$.  Then $\reg(R/(I,x)) \le \reg(R/I)$.
\end{lem}

We can now give a lower bound on regularity for certain square-free monomial ideals with one-dimensional hypergraphs.

\begin{prop}\label{Pmatch} 
Let $I \subset R$ be a square-free monomial ideal.  Let $\H = \H(I) = (V,X,E,\E)$.  Suppose that $\dim(\H) = 1$ and that there exist closed vertices $c_1,\ldots,c_t$ such that
\begin{enumerate}
\item For all $1 \le i, j \le t$, $c_i \not\in N_\H(c_j)$.
\item For every open vertex $v \in V$, there exists $1 \le i \le t$ such that $c_i \in N_\H(v)$.
\item For all $1 \le i \le t$, there is a unique $a_i \in X$ with $E_{a_i} = \{c_i\}$.
\end{enumerate}
Then
\[\reg(R/I) \ge |X| - |V|.\]
\end{prop}

\begin{proof} For $i = 1,\ldots,n$, set $I_i = I + (a_1,\ldots,a_i)$.  By Lemma~\ref{Ldhs},
\[\reg(R/I) \ge \reg(R/I_1) \ge \reg(R/I_2) \ge \cdots \ge \reg(R/I_n).\]
Set $\H_i = \H(I_i)$.  Note that at each step, say going from $I_i$ to $I_{i+1}$, we are removing one of the minimal generators of $I_i$ and replacing it by $a_i$.  No other generators are changed and all remain minimal.  Since $a_i$ was the unique variable with $E_{a_i} = \{c_i\}$, $I_i$ and $I_{i+1}$ have the same number of edges.  Since none of the closed vertices $c_1,\ldots,c_t$ were neighbors, we retain the fact that each of our chosen closed vertices has a unique label.  Hence the the number of vertices and number of variables in $\H_i$ is constant for all $i$.  

Now consider one of the open vertices $v \in V$ of $\H$.  By assumption, there was a closed vertex $c_i \in N_\H(v)$.  Hence there is an edge $F \in \E$, necessarily of size $2$, containing $v$ and $c_i$.  So $F = \{v, c_i\}$.  Let $\L = \{a \in X : E_a = F\}$.  In $\H_n = (V^{(n)},X^{(n)},E^{(n)},\E^{(n)})$, each of the elements $a \in \L$ satisfies $E^{(n)}_a = \{v\}$.  Hence there are no open vertices and $\H_n$ is a saturated hypergraph.  The result follows from Proposition~\ref{Psat}.
\end{proof}

If in addition, we assume all open vertices are isolated, we have the following formula for regularity.

\begin{cor}\label{Tmatch} Let $I \subset R = K[\A]$ be a square-free monomial ideal.  Let $\H = \H(I) = (V,X,E,\E)$.  Suppose that $\dim(\H) = 1$ and that there exist closed vertices $c_1,\ldots,c_t$ such that
\begin{enumerate}
\item For all $1 \le i< j \le t$, $c_i \not\in N_\H(c_j)$.
\item For every open vertex $v$, there exists $1 \le i \le t$ such that $c_i \in N_\H(v)$.
\item For all $1 \le i \le t$, there is a unique $a \in X$ with $E_a = \{c_i\}$.
\item All open vertices are isolated.
\end{enumerate}
Then
\[\reg(R/I) = |X| - |V|.\]
\end{cor}

\begin{proof} This follows directly from Proposition~\ref{Pmatch} and Theorem~\ref{Tiso}.
\end{proof}

\begin{eg} Let $R = K[a,b,\ldots,y,z]$ and let $I = (f_1,f_2,f_3,f_4,f_5,f_6)$ with $f_1 = aef, f_2 = bgh, f_3 = ei, f_4 = hk, f_5 = cgij, f_6 = dfjk$.  Set $\H = \H(I) = (V,X,E,\E)$, which is pictured in Figure~\ref{F7}.  Then we may take closed vertices $c_1 = 1$ and $c_2 = 2$, which are not neighbors of each other.  Moreover, every open vertex is isolated and is a neighbor of $1$ or $2$.  Applying Theorem~\ref{Tmatch} yields $\reg(R/I) = |X| - |V| = 11 - 6 = 5$.
\end{eg}

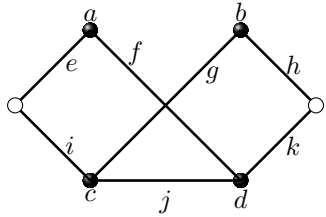
\begin{figure}[h]
\caption{$\H((aef,bgh,ei,hk,cgij,dfjk))$}\label{F7}
\begin{tikzpicture}
\usetikzlibrary{patterns}

\draw  [shape=circle] (2,0) circle (.1) ;
\draw  [shape=circle] (-2,0) circle (.1) ;
\shade [shading=ball, ball color=black] (-1,1) circle (.1)  node [above ]  { $a$};
\shade [shading=ball, ball color=black] (-1,-1) circle (.1) node [below] {$c$}; ;
\shade [shading=ball, ball color=black] (1,1) circle (.1) node [above ]  { $b$};
\shade [shading=ball, ball color=black] (1,-1) circle (.1) node [below ] { $d$};
\draw [line width=1pt  ] (-1,-1)--(1,-1)  node [pos=.5, below ] { $j$ } ;
\draw [line width=1pt  ] (1,1)--(-1,-1)  node [pos=.3, right ] { $g$ } ;
\draw [line width=1pt  ] (-1,1)--(1,-1)  node [pos=.3, above ] { $f$ } ;
\draw [line width=1pt  ] (1,1)--(1.93,.07)  node [pos=.5, right ] { $h$ } ;
\draw [line width=1pt  ] (1.93,-.07)--(1,-1)  node [pos=.5, right ] { $k$ } ;
\draw [line width=1pt  ] (-1.93,.07)--(-1,1)  node [pos=.5, right ] { $e$ } ;
\draw [line width=1pt  ] (-1.93,-.07)--(-1,-1)  node [pos=.5, right ] { $i$ } ;
\end{tikzpicture}
\end{figure}

It appears to be difficult to find a similar formula for even all one dimension hypergraphs that have no cycles.  The following example illustrates one of the difficulties one must deal with to extend the previous theorem.  

\begin{eg} Consider the ideals 
\[I = ideal(ab,bc,cde,ef,fghi,ij,jklm,mn,no) \text{ and  } J = (ab,bc,cdef,fg,ghi,ij,jklm,mn,no).\]
Both $\H(I)$ and $\H(J)$, pictured in Figure~\ref{F8}, are one-dimensional hypergraphs with the same number of edges and vertices.  Both $I$ and $J$ have the same number of generators in each degree.  It follows from Corollary~\ref{Tmatch} that $\reg(R/J) = |X| - |V| = 15 - 9 = 6$; however, Corollary~\ref{Tmatch} does not apply to $I$.  In fact, $\reg(R/I) = 5$.  
\end{eg}

\begin{figure}[h]
\caption{Two Similar Hypergraphs}\label{F8}
\begin{tikzpicture}
\path  (-1,2.5)--(1,2.5)  node [pos=.5, above ]{$\mathcal{H}(I)$};
\draw  [shape=circle] (-2.2,2) circle (.1) ;
\draw  [shape=circle] (-0.75,2) circle (.1) ;
\draw  [shape=circle] (0.75,2) circle (.1) ;
\draw  [shape=circle] (2.2,2) circle (.1) ;
\shade [shading=ball, ball color=black] (-3,0) circle (.1)  node [below ]  { $a$};
\shade [shading=ball, ball color=black] (-1.5,0) circle (.1) node [below] {$d$}; ;
\shade [shading=ball, ball color=black] (0,0) circle (.1) node [below ]  { $g,h$};
\shade [shading=ball, ball color=black] (1.5,0) circle (.1) node [below ] { $k,l$};
\shade [shading=ball, ball color=black] (3,0) circle (.1) node [below ] { $o$};
\draw [line width=1pt  ] (-3,0)--(-2.27,1.93)  node [pos=.5, left ] { $b$ }; 
\draw [line width=1pt  ] (-2.13,1.93)--(-1.5,0)  node [pos=.5, left ] { $c$ } ;
\draw [line width=1pt  ] (-1.5,0)--(-0.82,1.93)  node [pos=.5, left ] { $e$ } ;
\draw [line width=1pt  ] (-0.68,1.93)--(0,0)  node [pos=.5, left ] { $f$ }; 
\draw [line width=1pt  ] (0,0)--(0.68,1.93)  node [pos=.5, left ] { $i$ } ;
\draw [line width=1pt  ] (0.82,1.93)--(1.5,0)  node [pos=.5, left ] { $j$ } ;
\draw [line width=1pt  ] (1.5,0)--(2.13,1.93)  node [pos=.5, left ] { $m$ }; 
\draw [line width=1pt  ] (2.27,1.93)--(3,0)  node [pos=.5, left ] { $n$ } ;
\path  (7,2.5)--(9,2.5)  node [pos=.5, above ]{$\mathcal{H}(J)$};
\draw  [shape=circle] (5.8,2) circle (.1) ;
\draw  [shape=circle] (7.25,2) circle (.1) ;
\draw  [shape=circle] (8.75,2) circle (.1) ;
\draw  [shape=circle] (10.2,2) circle (.1) ;
\shade [shading=ball, ball color=black] (5,0) circle (.1)  node [below ]  { $a$};
\shade [shading=ball, ball color=black] (6.5,0) circle (.1) node [below] {$d,e$}; ;
\shade [shading=ball, ball color=black] (8,0) circle (.1) node [below ]  { $h$};
\shade [shading=ball, ball color=black] (9.5,0) circle (.1) node [below ] { $k,l$};
\shade [shading=ball, ball color=black] (11,0) circle (.1) node [below ] { $o$};
\draw [line width=1pt  ] (5,0)--(5.73,1.93)  node [pos=.5, left ] { $b$ }; 
\draw [line width=1pt  ] (5.87,1.93)--(6.5,0)  node [pos=.5, left ] { $c$ } ;
\draw [line width=1pt  ] (6.5,0)--(7.18,1.93)  node [pos=.5, left ] { $f$ } ;
\draw [line width=1pt  ] (7.32,1.93)--(8,0)  node [pos=.5, left ] { $g$ }; 
\draw [line width=1pt  ] (8,0)--(8.68,1.93)  node [pos=.5, left ] { $i$ } ;
\draw [line width=1pt  ] (8.82,1.93)--(9.5,0)  node [pos=.5, left ] { $j$ } ;
\draw [line width=1pt  ] (9.5,0)--(10.13,1.93)  node [pos=.5, left ] { $m$ }; 
\draw [line width=1pt  ] (10.27,1.93)--(11,0)  node [pos=.5, left ] { $n$ } ;
\end{tikzpicture}
\end{figure}
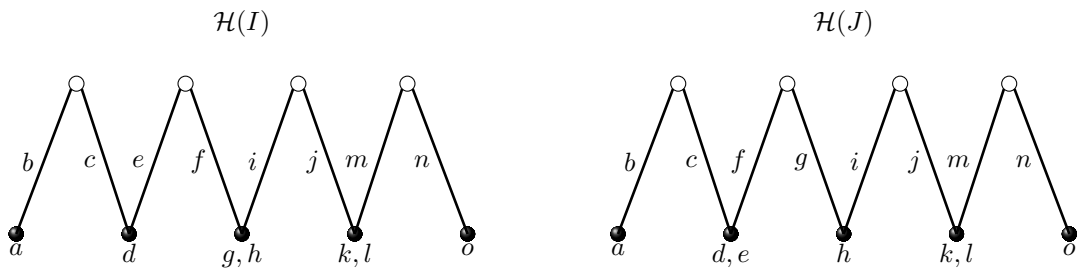

\section*{Acknowledgements} Part of this paper was completed while both authors were at MSRI and we thank personnel there for their support.  We also thank Alexandra Seceleanu who read an earlier draft of this paper.  Many of the computations in this paper were done with Macaulay2 \cite{m2}

\bibliographystyle{amsplain}
\bibliography{bib}

\end{document}